\documentclass{article} 
\author{Andy Hammerlindl and Ra\'ul Ures}
\title{Ergodicity and partial hyperbolicity on the 3-torus}

\usepackage{amsmath}
\usepackage{amssymb}
\usepackage{amsfonts}
\usepackage{amsthm}
\usepackage{graphicx}
\usepackage{setspace}
\usepackage{srcltx}

\newcommand{\R}{\mathbb{R}}
\newcommand{\Z}{\mathbb{Z}}
\newcommand{\T}{\mathbb{T}}
\newcommand{\bbR}{\mathbb{R}}
\newcommand{\bbZ}{\mathbb{Z}}
\newcommand{\bbT}{\mathbb{T}}

\newcommand{\Td}{\mathbb{T}^d}
\newcommand{\Rd}{\mathbb{R}^d}
\newcommand{\Zd}{\mathbb{Z}^d}

\newcommand{\F}{\mathcal{F}}
\newcommand{\Es}{E^s}
\newcommand{\Ec}{E^c}
\newcommand{\Eu}{E^u}

\newcommand{\Ws}{W^s}
\newcommand{\Wc}{W^c}
\newcommand{\Wu}{W^u}
\newcommand{\Wcu}{W^{cu}}
\newcommand{\Wcs}{W^{cs}}
\newcommand{\Wus}{W^{us}}
\newcommand{\inv}{^{-1}}

\newcommand{\Lf}{\mathcal{L}}

\newcommand{\Diff}{\operatorname{Diff}}
\newcommand{\Per}{\operatorname{Per}}

\newcommand{\length}{\operatorname{length}}

\newcommand{\id}{\operatorname{id}}
\newcommand{\supp}{\operatorname{supp}}
\newcommand{\gammahat}{\hat \gamma}

\newcommand{\normalize}[1]{\frac{#1}{\|#1\|}}

\newcommand{\catmap}{%
\ensuremath{(\begin{smallmatrix} 2&1 \\ 1&1 \end{smallmatrix})}%
}

\newtheorem{thm}{Theorem}[section]
\newtheorem{cor}[thm]{Corollary}
\newtheorem{lemma}[thm]{Lemma}
\newtheorem{prop}[thm]{Proposition}
\newtheorem{question}[thm] {\bf Question}
\newtheorem{conjecture}[thm] {\bf Conjecture}

\theoremstyle{remark}
\newtheorem*{remark} {\bf Remark}

\newtheorem*{notation} {\bf Notation}

\providecommand{\acknowledgement}{{\noindent\bf Acknowledgements}\quad}

\begin{document}

\maketitle


\begin{abstract}
    If a non-ergodic, partially hyperbolic diffeomorphism on the 3-torus is
    homotopic to an Anosov diffeomorphism $A$, it is topologically conjugate
    to $A$.
\end{abstract}
\section{Introduction} 
Ergodicity is commonplace among partially hyperbolic diffeomorphisms, and so
an interesting question is when such systems fail to be ergodic.  We study
this question in the specific case of the 3-torus, proving the following.

\begin{thm} \label{mainthm}
    Suppose $f:\T^3 \to \T^3$ is a conservative partially hyperbolic $C^2$
    diffeomorphism homotopic to a linear Anosov map $A:\T^3 \to \T^3$.   If
    $f$ is not ergodic, it is topologically conjugate to $A$.
\end{thm}
This research is motivated by the study of stable ergodicity.
In the space of conservative $C^2$ diffeomorphisms, an element $f$ is
\emph{stably ergodic} if every nearby diffeomorphism is also ergodic.
All Anosov diffeomorphisms are ergodic and they form an open subset of the
space of diffeomorphisms.  Therefore, every such system is an example
of a stably ergodic diffeomorphism.  In fact, for decades, these were the only
known examples.

Then, Grayson, Pugh, and Shub studied the time-one map of the geodesic flow
on a surface of constant negative curvature, showing it was also stably
ergodic \cite{grayson1994stably}.
While this diffeomorphism has some hyperbolic behaviour, it acts as an isometry
along the orbits of the flow.  In particular, it is an example of a
\emph{partially hyperbolic} system, a diffeomorphism $f:M \to M$ with an
invariant splitting of the tangent bundle $TM = \Eu \oplus \Ec \oplus \Es$, such
that vectors in $\Eu$ are exponentially expanded under iteration, those in $\Es$
are exponentially contracted, and where any expansion or contraction of the
\emph{center} bundle $\Ec$ is weak in comparison.  (The next section gives a
precise definition.)

Further research yielded a wealth of partially hyperbolic examples of stable
ergodicity, leading to the following conjecture of Pugh and Shub
\cite{pugh2000stable} \cite{pughshub_montevideo}.

\begin{conjecture} \label{pughshub1}
    Stable ergodicity is open and dense in the space of conservative, partially
    hyperbolic $C^2$ diffeomorphisms.
\end{conjecture}
This conjecture has been proved in a number of special cases, including
\cite{pugh2000stable}
\cite{RHRHTU-duke}
\cite{RHRHU-accessibility}
\cite{BW-annals}
\cite{bdp}
\cite{avila2008nonuniform}.
Recently, Avila, Crovisier, and Wilkinson announced a proof of a weaker
version of the conjecture where they show density, but only in the $C^1$
topology.
A key ingredient in the proofs of each of these results is a property called
accessibility.
A partially hyperbolic diffeomorphism is \emph{accessible} if any two points
$x,y \in M$ can be connected by a concatenation of paths, each path tangent
either to $\Eu$ or $\Es$.

Pugh and Shub split Conjecture \ref{pughshub1} into two subconjectures.

\begin{conjecture} \label{pughshub2}
    Accessibility holds on an open and dense subset of the space of partially
    hyperbolic diffeomorphisms (conservative or not).
\end{conjecture}
\begin{conjecture} \label{pughshub3}
    Every accessible conservative partially hyperbolic diffeomorphism is
    ergodic.
\end{conjecture}
As with \ref{pughshub1}, these conjectures have been proven in many special
cases.  See \cite{wilkinson2010conservative} for a recent survey.
In particular, the conjectures are true when the center bundle $\Ec$ is
one-dimensional.

\begin{thm}
    [Hertz-Hertz-Ures \cite{RHRHU-accessibility}]
    Accessibility is open and dense among conservative partially
    hyperbolic diffeomorphisms with one-dimensional center.
\end{thm}
\begin{thm}
    [Hertz-Hertz-Ures \cite{RHRHU-accessibility},
    Burns-Wilkinson \cite{BW-annals}] \label{accerg}
    Every conservative, accessible, partially hyperbolic $C^2$ diffeomorphism
    with one-dimensional center is ergodic.
\end{thm}
Thus, in this setting, the generic partially hyperbolic diffeomorphism is
ergodic and an interesting question is to describe the non-ergodic ones.
Even in the simplest case, where the manifold is three-dimensional and each of
the bundles $\Eu$, $\Ec$, and $\Es$ is one-dimensional, this is a difficult open
problem.

\begin{question}
    Which 3-manifolds support non-ergodic partially hyperbolic
    diffeomorphisms?
    Further, what special properties do these non-ergodic systems have?
\end{question}
As a motivating example, first consider a hyperbolic toral automorphism
$A:\bbT^2 \to \bbT^2$ such as the map induced by the matrix \catmap.
Define $f$ on $\bbT^2 \times [0,1]$ as $A$ times the identity map on $[0,1]$.
This is a non-ergodic partially hyperbolic diffeomorphism on a manifold with
boundary.
The stable $\Es$ and unstable $\Eu$ directions come from the hyperbolic map $A$
and are tangent to the surfaces $\bbT^2 \times \{t\}$.  The center direction is
tangent to fibers of the form $\{x\} \times [0,1]$.

To construct an example without boundary,
suppose $B:\bbT^2 \to \bbT^2$ is another toral automorphism that
commutes with $A$.  Then, the identification $(x,1) \sim (B x, 0)$ on $\bbT^2
\times [0,1]$ produces a closed manifold $M_B$ with a non-ergodic
partially hyperbolic diffeomorphism $f_B:M_B \to M_B$ coming from $f$.

If $B$ is the identity, $M_B$ is the 3-torus.
If $B$ is minus identity, that is, $B(x)=-x$ on $\bbT^2 = \bbR^2 / \bbZ^2$, then
$M_B$ is double covered by the 3-torus.
The only remaining possibility is that $B$ is hyperbolic, in which case $f_B$
can be thought of as the time-one map of an Anosov flow on $M_B$.

While these simple constructions do not give all possible examples of
non-ergodic partially hyperbolic \emph{diffeomorphisms} in dimension three,
they do give all of the \emph{manifolds} where such non-ergodic examples are
known to exist.  This fact and the results of \cite{RHRHU-nil} lead to the
following conjectures.

\begin{conjecture}
    [Hertz-Hertz-Ures \cite{RHRHU-nil}] \label{mfldB}
    If a 3-manifold $M$ supports a non-ergodic partially hyperbolic
    diffeomorphism, then $M = M_B$ (as defined above)
    where
    $B$ is $\pm \id$ or is hyperbolic.
\end{conjecture}
Each example $f_B$ contains embedded tori tangent to $\Eu \oplus \Es$.
Such tori are clear obstructions to accessibility and they occur in every
known non-ergodic example.

\begin{conjecture}
    [Hertz-Hertz-Ures] \label{anosovtori}
    For every non-ergodic partially hyperbolic diffeomorphism on a
    3-manifold, M, there is an embedded, periodic, incompressible torus
    tangent to $\Eu \oplus \Es$.
\end{conjecture}
Such a torus is an example of an \emph{Anosov torus} (as defined and studied in
\cite{RHRHU-tori}) and its existence implies that $M$ must be one of the $M_B$
discussed above.  In particular, if Conjecture \ref{anosovtori} is true, then
Conjecture \ref{mfldB} is also true as a consequence.

Suppose a diffeomorphism $f:M^3 \to M^3$ contains a torus $S=f^k(S)$ as in
Conjecture \ref{anosovtori}.
Since $S$ is incompressible, $\pi_1(S)$ injects into $\pi_1(M)$ and so $\pi_1(M)$
contains a copy of $\bbZ^2$ invariant under the group automorphism $f^k_*$.

If no such subgroup exists, then no such torus exists.  Unfortunately, for the
manifolds under consideration, this technique to rule out tori only works
in one specific case.

\begin{prop} \label{zz}
    Suppose $f$ is a diffeomorphism of $M_B$ where $B$ is $\pm \id$ or is
    hyperbolic.
    Then, exactly one of the following holds{:}
    \begin{itemize}
        \item $\pi_1(M_B)$ has an $f_*$-invariant subgroup isomorphic to
        $\bbZ^2$.
        \item $M_B = \bbT^3$ and $f_*$ is hyperbolic (when regarded as a $3
        \times 3$ matrix).
    \end{itemize}  \end{prop}
The proof is basic group theory and is left as an exercise.

This proposition, taken with Conjecture \ref{anosovtori}, suggests the
following conjecture, which was the main motivation in developing Theorem
\ref{mainthm}.

\begin{conjecture} \label{wanted}
    Suppose $f:\T^3 \to \T^3$ is a conservative partially hyperbolic $C^2$
    diffeomorphism homotopic to a linear Anosov map $A:\T^3 \to \T^3$.
    Then, $f$ is ergodic.
\end{conjecture}
Theorem \ref{mainthm} instead shows that under these assumptions, a system is
either ergodic, or is topologically conjugate to a well-understood, linear,
ergodic example.
This answers the question in spirit, but as the conjugacy to the Anosov system
may not be absolutely continuous, a counter example is still possible.  Such a
counterexample must be highly pathological in nature.

Suppose $f:\bbT^3 \to \bbT^3$ is a conservative, non-ergodic, partially
hyperbolic diffeomorphism homotopic to Anosov.  Then{:}
\begin{itemize}
    \item the conjugacy $h$ given by Theorem \ref{mainthm} satisfies
    $h(W_f^*)=W_A^*$ for $*=u,s,c$,
    \item the central Lyapunov exponent of $f$ is zero almost everywhere,
    \item $f$ fails to be 3-normally hyperbolic, that is, at some point
    $p \in \bbT^3$, the splitting fails to satisfy the inequality
    \[
        \|T_p f|_{\Es}\| < \|T_p f|_{\Ec}\|^3 < \|T_p f|_{\Eu}\|.
    \]  \end{itemize}
Note that by the second condition, the center must be \emph{non-uniformly}
close to an isometry, but by the third condition, it cannot be \emph{uniformly}
close to an isometry.
We give the precise statements and proofs of these three listed properties
in Sections \ref{conjsec}, \ref{expt}, and \ref{three-nh} respectively.

Under addition assumptions, similar results hold for higher-dimension tori.
These are described in Section \ref{highdim}.

\section{Definitions} \label{definitions} 
Functions $f$ and $g$ are \emph{topologically conjugate} if there is a
homeomorphism $h$ such that $f \circ h = h \circ g$.

A diffeomorphism of a manifold is \emph{conservative} if it preserves a finite
measure equivalent to Lebesgue.
A diffeomorphism $f:M \to M$ is \emph{ergodic} if it is conservative and any
$f$-invariant subset of $M$ either has zero measure or full measure.
For convenience, we take a \emph{non-ergodic} diffeomorphism to mean a
conservative diffeomorphism which is not ergodic.

A diffeomorphism $f$ on a compact Riemannian manifold $M$
is \emph{point-wise partially hyperbolic}
if there is a $Tf$-invariant splitting
$TM = \Es \oplus \Ec \oplus \Eu$
and functions $\sigma,\mu:M \to \bbR$ such that $\sigma < 1 < \mu$ and
\[
        \|Tf v^s\| < \sigma(p) < \|Tf v^c\| < \mu(p) < \|Tf v^u\|
\]
for all $p \in M$ and unit vectors
$v^s \in \Es_p$, $v^c \in \Ec_p$, and $v^u \in \Eu_p$. Further, $f$ is
\emph{absolutely partially hyperbolic} if the functions $\sigma$ and $\mu$ can be
taken to be constant.

The distinction between point-wise and absolute partially hyperbolicity is of
critical importance when studying systems on the 3-torus.  While there are
always unique foliations $\Wu$ and $\Ws$ tangent to $\Eu$ and $\Es$, the center
bundle $\Ec$ is not necessarily integrable.
A partially hyperbolic diffeomorphism is \emph{dynamically coherent} if there
are invariant foliations tangent to
$\Ec \oplus \Eu$ and $\Ec \oplus \Es$.

Brin, Burago, and Ivanov proved that every absolutely partially hyperbolic
diffeomorphism on the 3-torus is dynamically coherent \cite{BBI2}.  Soon after,
Rodriguez Hertz, Rodriguez Hertz, and Ures gave an example of a point-wise
partially hyperbolic system on $\bbT^3$ which is not dynamically coherent
\cite{RHRHU-nondyn}.
Further, Hammerlindl gave a classification result for absolutely partially
hyperbolic systems on the 3-torus \cite{ham-thesis}.  As a consequence of the
classification, these systems naturally fall into two distinct groups{:}
\begin{itemize}
    \item
    If $f$ is homotopic to Anosov, then every center leaf is dense in $\bbT^3$
    and is homeomorphic to a line.
    \item
    If $f$ is not homotopic to Anosov, then every center leaf is a circle.
    These circles form a trivial fiber-bundle over $\bbT^2$ and $f$ can be
    thought of as a skew-product.
\end{itemize}
As all known non-ergodic examples fall into the ``skew-product'' case, this
dichotomy provides further motivation for Conjecture \ref{wanted} and Theorem
\ref{mainthm}.

\begin{notation}
    Throughout this paper, ``partial hyperbolicity'' is taken to mean
    point-wise partial hyperbolicity unless the qualifier ``absolute'' is
    used.  In particular, Theorem \ref{mainthm} is proved in the point-wise
    case.
\end{notation}
\section{Outline and Externalities} 

The proof of Theorem \ref{mainthm} breaks into the following steps.
First, using results discovered for three-dimensional,
non-accessible systems, we show there is a foliation $\Wus$ tangent to $\Eu
\oplus \Es$.  By the work of Plante, associated to this foliation is a
holonomy invariant measure $\mu$, unique up to a constant factor.
This measure corresponds to an element of the cohomology $H^1(\T^3,\R)$, and as
$f$ acts hyperbolically on the cohomology,
$f^* \mu = \lambda \mu$ for some $\lambda < 1$.  Then $\mu(f^n \circ \gamma) \to 0$ as
$n \to \infty$ for any curve $\gamma$ transverse to $\Wus$ which implies that $f$
is topologically contracting in the center direction $\Ec$.  From this, we
deduce that $f$ is expansive.  The work of Vieitez then shows that $f$ is
conjugate to Anosov.

In the next section, we assume throughout that $f:\T^3 \to \T^3$ is a partially
hyperbolic system homotopic to Anosov.  To avoid confusion, we list in
advance the general theorems used.

Given a diffeomorphism $f:M \to M$, an injectively immersed submanifold $S
\subset M$ has \emph{Anosov dynamics} if $f^k(S) = S$ for some non-zero integer
$k$ and $f^k|_S$ is Anosov.  We say further that $S$ has Anosov dynamics
\emph{with dense periodic points} if $\Per(f^k|_S)$ is dense in the intrinsic
topology of $S$.

\begin{thm}
    [Hertz-Hertz-Ures \cite{RHRHU-nil}] \label{nonacc}
    Let $f:M \to M$ be a conservative partially hyperbolic diffeomorphism of an
    orientable 3-manifold $M$. Suppose that the bundles $E^*$ are also
    orientable, $* = s, c, u$, and that $f$ is not accessible. Then one of
    the following possibilities holds{:}
    \begin{enumerate}
        \item
        there is an $f$-periodic incompressible torus tangent to $\Eu \oplus \Es$;
        \item
        there is an $f$-invariant lamination $\varnothing  \ne  \Gamma(f)  \ne  M$
        tangent to $\Eu \oplus \Es$ that trivially extends to a (not necessarily
        invariant) foliation without compact leaves of $M$. Moreover, each
        boundary leaf of $\Gamma(f)$ has Anosov dynamics with
        dense periodic points;
        \item
        there is a Reebless invariant foliation tangent to $\Eu \oplus \Es$.
    \end{enumerate}  \end{thm}
For our problem domain, we show that the third case is the only case possible.
Then, we may use the results of Novikov compact leaf theory, as was
generalized to the $C^0$ case by Solodov \cite{solodov1984components}.

\begin{prop} \label{reebless}
    If $\F$ is a Reebless codimension one $C^0$ foliation $\F$ of
    a 3-manifold $M$, then
    \begin{itemize}
        \item there is no closed null-homotopic curve transverse to $\F$, and
        \item for every leaf $L$, the induced map $\pi_1(L) \to \pi_1(M)$ is
        injective.
    \end{itemize}  \end{prop}
A key intermediate in proving Theorem \ref{nonacc} is the following, which
will be used specifically in the next section.

\begin{prop}
    [Hertz-Hertz-Ures \cite{RHRHU-nil}] \label{subsetdynamics}
    Let $f:M \to M$ be conservative and partially hyperbolic with
    one-dimensional center.
    If $\Lambda$ is a closed, $f$-invariant subset of $M$
    consisting of leaves tangent to $\Eu \oplus \Es$, then every component of
    $\partial \Lambda$ is a leaf having Anosov dynamics with dense periodic
    points.
\end{prop}
These boundary leaves also satisfy the following.

\begin{prop}
    [Franks \cite{Franks1}] \label{udense}
    If $f:S \to S$ is an Anosov diffeomorphism and $\Per(f)$ is dense in $S$,
    then for any periodic point $x \in S$, the unstable manifold $\Wu(x)$ is
    dense in $S$.
\end{prop}
This is a restatement of (1.7) and (1.8) as given in \cite{Franks1}.  Note that
the proofs only require $S$ to be connected, not necessarily compact.

The work of Plante shows that many codimension one foliations give rise to
holonomy invariant measures.  The following is a combination of (4.1) and
(7.2) as stated in \cite{plante1975foliations}.

\begin{prop}
    [Plante \cite{plante1975foliations}] \label{muexists}
    Let $M$ be a compact manifold such that $\pi_1(M)$ has non-exponential
    growth, and let $\F$ be a codimension one foliation.  If $L$ is a leaf
    which does not intersect any null-homotopic closed transversal, then there
    is a holonomy invariant measure with support equal to the closure of $L$.
\end{prop}

The following is (8.5) from the same paper.

\begin{prop}
    [Plante \cite{plante1975foliations}]
    Let $\F$ be a codimension one foliation of class $C^r$ ($r  \ge  0$) of a
    compact manifold $M$.  If $\mu$ is an $\F$-invariant measure then there is
    a unique decomposition of $\mu$,
    \[
        \mu = \mu_K + \mu_1 + \cdots + \mu_n
    \]
    such that the following hold:
    \begin{enumerate}
        \item $\supp \mu_K$ is a union of compact leaves.
        \item $\supp \mu_i$ is connected and is a union of non-compact leaves,
        $i=1, \cdots, n$.
        \item The sets $\supp \mu_i$, $i = 1, \cdots, n$ are pairwise disjoint.  \end{enumerate}
    Furthermore, if $\F$ is transversely oriented, $n  \le  H_1(M; \R)/2$.
\end{prop}
For a foliation $\F$ of $M$, let a \emph{minimal set} signify a closed
non-empty $\F$-saturated subset which contains no proper subset with the same
properties.

\begin{cor} \label{finitesupports}
    If $\F$ is a codimension one foliation without compact leaves on a compact
    manifold $M$, there are at most a finite number of subsets of the form $X
    \subset M$ such that $X$ is a minimal set and $X = \supp \mu$ for some
    holonomy invariant measure $\mu$.
\end{cor}
Further, the measures supported on minimal sets are unique up to proportion, as
demonstrated in the book of Hector and Hirsch
(see Chapter X Theorem 2.3.3 of \cite{hector-hirsch-partb}).

\begin{thm}
    [Hector-Hirsch \cite{hector-hirsch-partb}] \label{hhmeasure}
    Let $\F$ be a codimension one foliation.  Let $\mu$ be $F$-invariant with
    support a minimal set of $\F$ which is not a compact leaf.  If $\mu'$ is
    another $\F$-invariant measure with equal support, there is $c \in \R$ such
    that $\mu' = c \mu$.
\end{thm}
Once we establish expansiveness, the final step is to invoke the following
result of Vieitez.

\begin{thm}
    [Vieitez \cite{vieitez2002lyapunov}] \label{vieitez}
    Let $M$ be a three-dimensional compact connected oriented manifold and
    $f : M \to M$ an expansive diffeomorphism.  If $NW(f) = M$ then $f$ is
    conjugate to a linear Anosov diffeomorphism and $M = \T^3$.
\end{thm}
\begin{remark}
    Here, $NW(f) = M$ is an assumption equivalent to saying that for every
    non-empty open subset $U \subset M$, there is $k  \ne  0$ such that $U$
    intersects $f^k(U)$.  This holds for all conservative diffeomorphisms by
    Poincar\'e recurrence.
    While we assume throughout the next section that $f$ is conservative, we
    only need this assumption to use that $NW(f) = M$ and to apply Theorem
    \ref{accerg}.
    In fact, Theorem \ref{nonacc} and Proposition \ref{subsetdynamics} hold in
    the non-conservative case, so long as $NW(f) = M$.  Therefore, Theorem
    \ref{mainthm} also holds in the non-conservative case if the condition
    ``not ergodic'' is replaced by ``not accessible and $NW(f) = \bbT^3$.''
\end{remark}
\section{The proof} \label{mainproof} 
In this section, assume $f:\T^3 \to \T^3$ is a conservative, non-accessible,
partially hyperbolic diffeomorphism homotopic to Anosov.  In particular, the
action of $f$ on $\pi_1(\T^3) \cong \Z^3$ is as a hyperbolic linear map.
By lifting to a finite cover, assume that the bundles $E^u$, $E^c$, and $E^s$
are orientable.  The original map is expansive if and only if its lift to this
finite cover is.

As $f$ is homotopic to Anosov, it forbids certain invariant subsurfaces.

\begin{lemma} \label{noad}
    There is no injectively immersed surface $i:S \to M$ such that
    $i_*:\pi_1(S) \to \pi_1(M)$ is injective and $S$ has Anosov dynamics with
    dense periodic points.
\end{lemma}
\begin{proof}
    Suppose there is such a leaf $S$.  As $f^k(S)=S$, the image of $\pi_1(S)$
    in $\pi_1(\T^3)$ is an $f^k_*$-invariant subgroup.  As $f^k_*$ is
    hyperbolic, the only possibilities for such an invariant subgroup are the
    trivial group or a full rank subgroup.  As no surface has fundamental
    group isomorphic to $\Z^3$, $S$ must be simply connected.  As the 2-sphere
    does not permit Anosov dynamics, $S$ must be a plane.

    By Proposition \ref{udense}, there is a dense unstable leaf through $S$.
    In the case of a plane, however, if an unstable manifold passes near
    itself, then by connecting the two nearby
    segments of the unstable manifold, one can construct a trapping region, a
    Jordan curve transverse to $\Eu$ such that the orientation of $\Eu$ points
    either entirely in or out of the bound region.  This contradicts the fact
    that the leaf is dense.
\end{proof}
\begin{prop} \label{usint}
    There is a unique $f$-invariant foliation tangent to $\Eu \oplus \Es$ without
    compact leaves.
\end{prop}
\begin{proof}
    We first rule out two of the cases of Theorem \ref{nonacc}.
    Suppose there is an $f$-periodic incompressible torus $S$ tangent to $\Eu
    \oplus \Es$.  Then, $\pi_1(S) \cong \Z^2$ injects into $\pi_1(\T^3)$
    and the image is invariant under $f^k_*$ for some $k$.  As $f^k_*$ is
    hyperbolic, no such rank two subgroup exists, giving a contradiction.

    The second case of Theorem \ref{nonacc} implies a foliation without compact
    leaves, some of which have Anosov dynamics with dense periodic points.  As
    the foliation is Reebless, the inclusion of any leaf $\Lf \subset \T^3$ is
    $\pi_1$-injective and Lemma \ref{noad} gives a contradiction.

    Thus, only the third case of Theorem \ref{nonacc} is possible, and there is
    a Reebless foliation tangent to $\Eu \oplus \Es$.  As $\Wu$ and $\Ws$ are
    uniquely integrable, the foliation is unique.

    Suppose this foliation has a compact leaf $S$ which, as the leaf is
    foliated by $\Wu$ and $\Ws$, must be a 2-torus.  Further, as the foliation
    is Reebless, $\pi_1(S)$ injects into $\pi_1(\T^3)$.  As $f_* \pi_1(S)  \ne 
    \pi_1(S)$, we can find a closed loop $\gamma \subset f(S)$ with homotopy
    class $[\gamma] \in \pi_1(\T^3) \setminus \pi_1(S)$.  By intersection number
    arguments, $\gamma$ must intersect $S$, but by the unique integrability, we
    have that $\gamma$ is contained in $S$, a contradiction.
\end{proof}
Let $\Wus$ denote the foliation tangent to $\Eu \oplus \Es$.

\begin{lemma} \label{minsupport}
    For every minimal set $\Lambda$ of $\Wus$, there is a holonomy invariant
    measure with support equal to $\Lambda$.
\end{lemma}
\begin{proof}
    The fundamental group $\pi(\T^3) \cong \Z^3$ has non-exponential growth, and
    as $\Wus$ is Reebless, there is no null-homotopic closed transversal by
    Proposition \ref{reebless}.
    Therefore, if $L$ is a leaf in a minimal set $\Lambda$, by Proposition
    \ref{muexists}, there is a holonomy invariant measure with support equal to
    $\overline L = \Lambda$.
\end{proof}
\begin{prop} \label{usmin}
    The foliation $\Wus$ is minimal; that is, $\T^3$ is the only minimal set.
\end{prop}
\begin{proof}
    Suppose $\Lambda$ is a minimal set. As the foliation $\Wus$ is $f$-invariant,
    $f^k(\Lambda)$ is also a minimal set for every integer $k$.  By Corollary
    \ref{finitesupports} and Lemma \ref{minsupport}, there is $k$ such that
    $f^k(\Lambda)=\Lambda$.  By Proposition \ref{subsetdynamics}, using $f^k$ in
    place of $f$, every leaf in $\partial \Lambda$ has Anosov dynamics with
    dense periodic points, and so by Lemma \ref{noad}, $\partial \Lambda =
    \varnothing$, which is only possible if the (non-empty) minimal set
    $\Lambda$ is all of $\T^3$.
\end{proof}

Now, fix $\mu$ such that $\supp \mu = \T^3$. As it has full support, $\mu(\gamma) >
0$ for any positive length curve $\gamma$ transverse to $\Wus$.

\begin{prop}
    There is $\lambda  \ne  1$ such that $f^* \mu = \lambda \mu$.
\end{prop}
\begin{proof}
    By Proposition \ref{hhmeasure}, $\mu$ is unique up to a constant.
    As the pullback defined by $f^* \mu(\gamma) = \mu(f \circ \gamma)$ is another
    non-zero holonomy-invariant measure, there is $\lambda > 0$ such that $f^*
    \mu = \lambda \mu$.

    Any path on $\T^3$ is homotopic to a path consisting of a concatenation of
    segments, each either tangent to $\Wus$ or transverse.  As the foliation is
    transversely orientable, this canonically associates $\mu$ with a non-zero
    element $[\mu]$ of the cohomology group $H^1(\T^3,\R)$.  (See
    \cite{plante1975foliations} or \cite{hector-hirsch-partb} for details.)
    Further,
    \[
        \lambda [\mu] = [\lambda \mu] = [f^* \mu] = \pm f^*[\mu].
    \]
    That is, $\pm \lambda$ is an eigenvalue of the linear map $f^*$ on
    $H^1(\T^3,\R) \cong \R^3$.  As $f$ is homotopic to Anosov, the map $f^*$ is
    hyperbolic, and $\lambda  \ne  1$.
\end{proof}
By replacing $f$ by $f \inv$ if necessary, assume for the remainder of the
section that $f^* \mu = \lambda \mu$ where $\lambda < 1$.

\begin{lemma}
    For $\epsilon>0$ there is $\delta>0$ such that for any curve $\gamma$ tangent
    to $\Ec$,
    \[
        \mu(\gamma) < \delta   \quad \Rightarrow \quad   \length(\gamma) < \epsilon.
    \]  \end{lemma}
\begin{proof}
    Suppose instead that $\{\gamma_k\}$ is a sequence of curves $[0,1] \to \T^3$
    such that $\mu(\gamma_k)$ goes to zero, but $\length(\gamma_k)$ does not.
    Take a finite atlas of foliation charts for $\Wus$ covering the manifold,
    and by the Lebesgue Covering Lemma fix $\rho>0$ such that every ball of
    radius $\rho$ is contained in the domain of one of the charts.  Without
    loss of generality, by replacing each $\gamma_k$ by a subcurve, which only
    decreases its measure, assume $\length \gamma_k < \rho$ for all $k$.  By
    restricting to a subsequence, further assume that all of the $\gamma_k$ lie
    in the domain $U$ of one of the foliation charts, and that the sequences
    of endpoints $\{\gamma_k(0)\}$ and $\{\gamma_k(1)\}$ converge.

    There is a center curve, $J$, through $U$ such that the $\Wus$ holonomy
    inside the foliation box defines a retract $r:U \to J$.  By transversality
    of $\Ec$ to $\Wus$, it follows that
    \[
        \lim_k r(\gamma_k(0))  \ne  \lim_k r(\gamma_k(1)).
    \]
    Therefore, $J' = \bigcup_{n=1}^\infty \bigcap_{k=n}^\infty r(\gamma_k([0,1]))$
    is a non-trivial subinterval of $J$. Then
    \[
        \mu(J')
         \le  \lim \inf \mu(r(\gamma_k([0,1])))
        = \lim \inf \mu(\gamma_k([0,1]))
        = 0.
    \]
    This contradicts the fact that $\mu$ has full measure.
\end{proof}
\begin{lemma}
    For $\epsilon>0$ there is $\delta>0$ such that for any curve $\gamma$ tangent
    to $\Ec$,
    \[
        \length(\gamma) < \delta   \quad \Rightarrow \quad   \mu(\gamma) < \epsilon.
    \]  \end{lemma}
The proof is similar to the previous lemma, and is omitted.

\begin{cor} \label{ccont}
    For $r,R>0$ there is an integer $N$ such that for any curve $\gamma$
    tangent to $\Ec$ and any $n>N$,
    \[
        \length(\gamma) < R  \quad \Rightarrow \quad  \length(f^n \circ \gamma) < r.
    \]  \end{cor}
\begin{proof}
    This is a straightforward consequence of the previous two lemmas, and the
    fact that $\mu(f^n \circ \gamma) = \lambda^n \mu(\gamma)$.
\end{proof}
If the center bundle $\Ec$ were uniquely integrable, it would be easy to prove
from the topological contraction in the center direction that $f$ is
expansive.  However, technical issues arise in trying to prove unique
integrability, so we continue the proof without it.  First is a form
of expansiveness in the $\Ec \oplus \Es$ direction.

\begin{lemma} \label{csexp}
    There is a constant $\delta > 0$ such that the following holds.
    If $\alpha^s, \alpha^c : [0,1] \to \T^3$ are
    paths tangent to $\Es$ and $\Ec$ respectively where
    $\alpha^s(1) = \alpha^c(0)$,
    $x = \alpha^s(0)$, and
    $y = \alpha^c(1),$
    then $x=y$ if and only if $d(f^n(x), f^n(y)) < \delta$ for all $n \in \Z$.
\end{lemma}
\begin{proof}
    Fix $r>0$ sufficiently small.  By transversality of $\Es$ and $\Ec$,
    for any $a \in (0,1)$, there is $\delta > 0$ such that the following
    property holds.
    \begin{quote}
        If $\gamma^s,\gamma^c:[0,1] \to \T^3$ are curves tangent to $\Es$ and $\Ec$
        respectively, $\gamma^s(1) = \gamma^c(0)$, and
        $
        a r < \max\{\length \gamma^s, \length \gamma^c\} < r
        $
        then $d(\gamma^s(0), \gamma^c(1)) > \delta$.
    \end{quote}
    Fix $a \in (0,1)$ such that $\length(f \circ \gamma) > a \length (\gamma)$ for
    any positive length $C^1$ path $\gamma:[0,1] \to \T^3$.  This defines $\delta$.

    Now suppose $\alpha^s$, $\alpha^c$, are as in the statement of the lemma,
    and that $x=\alpha^s(0) \ne \alpha^c(1)=y$.  We want to find $n \in \Z$ such
    that $d(f^n(x), f^n(y)) > \delta$.
    If $\alpha^c$ or $\alpha^s$ has length zero, this task is easy, so assume
    both are of positive length.  Then,
    $\mu(f^n \circ \alpha^c) = \lambda^n \mu(\alpha^c)$
    implies that $\length(f^n \circ \alpha^c)$ goes to zero as $n \to +\infty$ and
    goes to infinity as $n \to -\infty$.  The definition of partial
    hyperbolicity implies the same limits for $\length(f^n \circ \alpha^s)$.
    Thus, there is $n \in \Z$ such that
    \[
        \max\{\length(f^n \circ \alpha^s), \length(f^n \circ \alpha^c)\} > r
    \]
    and
    \[
        \max\{\length(f^{n+1} \circ \alpha^s), \length(f^{n+1} \circ \alpha^c)\} < r.
    \]
    The constant $a$ was defined such that
    \[
        a r <
        \max\{\length(f^{n+1} \circ \alpha^s), \length(f^{n+1} \circ \alpha^c)\} < r
    \]
    which implies that $d(f^{n+1}(x), f^{n+1}(y)) > \delta$.
\end{proof}
\begin{lemma} \label{lps}
    For $\epsilon < 0$, there is $\delta < 0$ such that if $x,y \in \T^3$ satisfy
    $d(x,y) < \delta$, then there are curves $\gamma^\sigma:[0,1] \to \T^3$ for
    $\sigma=c,s,u$ each of length less than $\epsilon$ and such that the
    concatenation $\gamma^c \gamma^s \gamma^u$ is a path from $x$ to $y$.

    Moreover, the point $\gamma^s(1)=\gamma^u(0)$ is unique.  To be precise, if
    $\alpha^c \alpha^s \alpha^u$ is another such triple of paths, then
    \[
        \alpha^s(1)=\alpha^u(0)=\gamma^s(1)=\gamma^u(0).
    \]  \end{lemma}
\begin{proof}
    The existence part of the proof follows from the transversality of the
    subbundles.  Here, we prove the uniqueness claimed above.  Suppose
    $\epsilon>0$ is given.
    As $\Wu$ is expanded by $f$, there is $R>0$ such that if $x$ and $y$ are
    distinct points on the same leaf of $\Wu$, then
    \[
        d(f^n(p), f^n(q)) > R
    \]
    for some (possibly negative) integer $n$.  For a given integer $N$, there
    is $r>0$ such that if $d_u(p,q) < r$ then $n$ can be taken greater than
    $N$.

    Using Corollary \ref{ccont}, fix $N$ such that for any $n>N$ and any
    curve $\gamma$ tangent either to $\Ec$ or $\Es$
    \[
        \length(\gamma) < \epsilon  \quad \Rightarrow \quad  \length(f^n \circ \gamma) < \frac{R}{4}
    \]
    for all $n>N$.
    Once $N$ is fixed, fix $r > 0$ such that
    \[
        d_u(p,q) < r  \quad \Rightarrow \quad  d(f^n(p), f^n(q)) > R
    \]
    for some $n>N$.

    Now set $\epsilon' = \min(r/2, \epsilon)$ and let $\delta > 0$ be the
    corresponding constant in the existence portion of this lemma.
    Suppose $x,y \in \T^3$ are points such that $d(x,y) < \delta$ and
    $\gamma^c \gamma^s \gamma^u$ and
    $\alpha^c \alpha^s \alpha^u$ are triples of paths of length at most
    $\epsilon'$.
    Let $p = \gamma^s(1) = \gamma^u(0)$ and $q = \alpha^s(1) = \alpha^u(0)$.
    Then, as $p$ and $q$ are connected by the concatenation of four paths of
    length at most $\epsilon$ tangent either to $\Es$ or $\Ec$, it follows that
    $d(f^n(p),f^n(q)) < R$ for all $n>N$.  However, as $p$ and $q$ are also
    connected by the concatenation of two paths of length at most $r/2$
    tangent to $\Eu$, if $p  \ne  q$, it follows that $d(f^n(p), f^n(q)) > R$ for
    some $n>N$.  Thus $p$ must equal $q$ and the desired uniqueness is proved.
\end{proof}
\begin{prop}
    $f$ is expansive.
\end{prop}
\begin{proof}
    Let $\epsilon_1 > 0$ be small enough that for any two distinct points $p,q$
    on the same unstable leaf, there is $n \in \Z$ such that
    $d(f^n(p),f^n(q))>\epsilon_1$.  Let $\epsilon_2 > 0$ be small enough that
    \[
        \length \gamma < \epsilon_2  \quad \Rightarrow \quad  \length f \circ \gamma < \epsilon_1
    \]
    for any $C^1$ path in $\T^3$.  Let $\delta_1$ and $\delta_2$ be the
    corresponding constants given by Lemma \ref{lps}.  Set $\delta>0$ such that
    it satisfies the conditions of Lemma \ref{csexp} and is smaller than both
    $\delta_1$ and $\delta_2$.

    Now suppose $x_0,y_0 \in M$ are such that $d(x_n,y_n) < \delta$ for all $n
    \in Z$ where $x_n = f^n(x_0)$ and $y_n = f^n(y_0)$.  By Lemma \ref{lps},
    there are paths $\gamma_n^c$, $\gamma_n^s$, $\gamma_n^u$ of length at most
    $\min(\epsilon_1,\epsilon_2)$
    connecting $x_n$ and $y_n$.  Let $z_n$ denote $\gamma_n^s(1)$.  Then as
    $f \circ \gamma_n^c$, $f \circ \gamma_n^s$, $f \circ \gamma_n^u$ are paths of
    length at most $\epsilon_1$ connecting $f(x_n)=x_{n+1}$ to
    $f(y_n)=y_{n+1}$, it follows by uniqueness, that $f(z_n) = z_{n+1}$ for
    all $n \in \Z$.

    By the choice of $\epsilon_1$ at the start of the proof, $z_n = y_n$ for all
    $n$, and then by Lemma \ref{csexp} and the choice of $\delta$, $x_n = y_n$
    for all $n$, proving expansiveness.
\end{proof}

\section{Higher dimensions} \label{highdim} 

We now prove a version of Theorem \ref{mainthm} in higher dimensions under
additional assumptions.
For the motivation behind these assumptions, see \cite{ham-thesis}.

\begin{thm} \label{highdimthm}
    Suppose $f:\Td \to \Td$ is a conservative partially hyperbolic $C^2$
    diffeomorphism homotopic to a linear Anosov map $A:\Td \to \Td$.
    Further, suppose
    \begin{itemize}
        \item $f$ is absolutely partially hyperbolic,
        \item $f$ has one-dimensional center, and
        \item the foliations $\Wu$ and $\Ws$ of $f$ are quasi-isometric.
    \end{itemize}
    If $f$ is not ergodic, it is topologically conjugate to $A$.
\end{thm}
Here, a foliation $W$ on a manifold $M$ is \emph{quasi-isometric} if, after
lifting to the universal cover $\tilde M$, there is a constant $Q$ such that
$d_{\tilde W}(x,y) < Q d_{\tilde M}(x,y) + Q$
for all $x$ and $y$ on the same leaf of the lifted foliation $\tilde W$.

For the remainder of this section, assume $f$ satisfies the assumptions of
Theorem \ref{highdimthm} and is not ergodic.

\begin{prop} \label{noanosov}
    No submanifold $S$ tangent to $\Eu \oplus \Es$ has Anosov dynamics with dense
    periodic points.
\end{prop}
\begin{proof}
    Let $\tilde S \subset \Rd$ be a connected component of $P \inv(S)$ where $P:\Rd
    \to \Td$ is the universal covering.  There is a lift $\tilde f:\Rd \to \Rd$ of
    $f$ such that ${\tilde f}^k(\tilde S) = \tilde S$.  For simplicity, we assume
    $k=1$.

    For $x \in \Rd$, consider the set
    $\bigcup_{y \in \Wu(x)} \Ws(y)$.
    As shown in \cite{ham-thesis}, this set, called the \emph{pseudoleaf}
    through $x$, is a
    complete properly-embedded topological hyperplane in $\Rd$, which
    intersects every center leaf exactly once.  If $x$ is in
    $\tilde S$, then the pseudoleaf through $x$ must be contained in $\tilde S$,
    and then, by virtue of completeness, the two submanifolds must coincide.
    In particular, every center leaf intersects $\tilde S$ exactly once.

    Let $A:\Rd \to \Rd$ be the hyperbolic linear map to which the system is
    homotopic.  That is, the unique linear map $A$ such that
    $\|A(x) - \tilde f(x)\|$ is bounded for all $x \in \Rd$.
    Let $h:\Rd \to \Rd$ be a leaf conjugacy from $A$ to $\tilde f$ as constructed
    in \cite{ham-thesis}.  That is, $A$ may be viewed as partially hyperbolic
    and $h$ is a homeomorphism such that for any center leaf $\Lf$ of $A$,
    $h(\Lf)$ is a center leaf of $\tilde f$ and
    \[
        \tilde f h(\Lf) = h A (\Lf).
    \]
    Further, $h$ satisfies the relation $h(x+z)=h(x)+z$ for all $x \in \Rd$ and
    $z \in \Zd$.

    Let $V \subset \Rd$ denote the codimension one subspace of $\Rd$ spanned by
    the stable and unstable subspaces of the partially hyperbolic splitting of
    $A$.  Then, $A(V) = V$ and each center leaf of $A$ intersects $V$ in
    exactly one point.

    Suppose $z \in \Zd$ is such that $\tilde S = \tilde S + z$.
    Then
    $x + k z \in \tilde S$ for any $x \in \tilde S$ and $k \in \bbZ$.
    As $\tilde S$ is a $us$-pseudoleaf, \cite[Corollary 2.9]{ham-thesis}
    implies that if $z \ne 0$ then the sequence
    \[
        \normalize{(x + k z) - x} = \normalize{z}
    \]
    tends to $\Es_A \oplus \Eu_A$ as $k \to \infty$.
    In other words, $z \in V$ and $V = V + z$.

    Define $H: V \to \tilde S$ by requiring
    $H(\Lf \cap V)=h(\Lf) \cap \tilde S$ for each center leaf $\Lf$ of $A$.
    This uniquely determines $H$ and one can verify that it is a true
    conjugacy, a homeomorphism such that
    \[
        \tilde f H(x) = H A (x)
    \]
    for all $x \in V$.


    Suppose $x \in \tilde S$ projects to a periodic point $P(x)=f^k(P(x))$.
    Then, there is $z \in \Zd$ such that $x+z=\tilde f^k(x) \in \tilde S$.
    Let $y \in V$ be such that $H(y)=x$.  Let $\Lf$ be the center leaf through
    $y$.  Note that
    \[
        h A^k(\Lf) = \tilde f^k h(\Lf) = h(\Lf) + z = h(\Lf + z)
    \]
    and as $h$ is a homeomorphism, $A^k(\Lf) = \Lf + z$.  Using that $V=V+z$,
    \[
        A^k(\Lf \cap V) = A^k(\Lf) \cap A^k(V) = (\Lf + z) \cap (V + z)
    \]
    and therefore $A^k(y) = y+z$.

    We have shown that if $x \in \tilde S$ projects to an $f$-periodic point on
    $S$, then $y = H \inv(x)$ projects to an $A$-periodic point on $P(V)$.
    This implies that $A$ has dense periodic points on the invariant
    submanifold $P(V) \subset \Td$.

    The submanifold $P(V)$ can be viewed as $\T^j \times \R^k$ where $j$ is
    the rank of $\Zd \cap V$ and $j+k=d-1$.
    Further, $A|_{P(V)}$ can be written as
    \[
        (x,y) \mapsto (A_1(x) + B(y), A_2(y))
    \]
    where $A_1:\T^j \to \T^j$ is a hyperbolic toral automorphism,
    $A_2:\R^k \to \R^k$ is a hyperbolic linear map, and $B: \R^k \to \T^j$ is
    linear.

    The only way this map can have dense periodic points is if $A_2$ is
    trivial.  That is, $k=0$ and $A|_{P(V)} = A_1$.
    As both $A$ and $A_1$ are linear toral automorphisms, they have
    determinants equal to $\pm 1$ when viewed as linear maps.
    As the eigenvalues of $A_1$ consist of all but one of the eigenvalues of
    $A$, it would imply that the remaining eigenvalue is $\pm 1$, contradicting
    the standing assumption that $A$ is hyperbolic.
\end{proof}
\begin{prop}
    $\Eu \oplus \Es$ is uniquely integrable.
\end{prop}
\begin{proof}
    Let $AC(x)$ denote the accessibility class of $x$, all points reachable
    from $x$ by taking a concatenation of paths, each tangent to $\Eu$ or to
    $\Es$.  Define
    \[
        \Gamma(f) = \{ x \in \Td : AC(x) \text{ is not open}\}.
    \]
    One can verify that $\Gamma(f)$ is closed and $f$-invariant, and if $f$ is
    not accessible then $\Gamma(f)  \ne  \varnothing$.  Somewhat less trivially,
    $\Gamma(f)$ is laminated by leaves tangent to $\Eu \oplus \Es$
    \cite{RHRHU-accessibility}.  By Proposition \ref{subsetdynamics}, any leaf
    of $\partial \Gamma(f)$ has Anosov dynamics with dense periodic points,
    therefore, by the previous proposition, $\partial \Gamma(f) = \varnothing$,
    and (using that $f$ is not accessible) $\Gamma(f)$ is all of $\Td$.
\end{proof}
\begin{prop}
    There is no null-homotopic loop transverse to $\Eu \oplus \Es$.
\end{prop}
\begin{proof}
    Consider the unique foliation $\Wus$ tangent to $\Eu \oplus \Es$ on the
    universal cover $\Rd$.
    As with $\tilde S$ in the proof of
    Proposition \ref{noanosov}, every leaf of $\Wus$ is a complete properly
    embedded topological hyperplane which intersects every center leaf exactly
    once.  From this, one can see that there is no closed loop on $\Rd$
    transverse to $\Wus$ and therefore no null-homotopic transverse loop on
    $\Td$.
\end{proof}
These last three propositions replace all of the techniques specific to the
three-dimensional case that were used in the previous section.  Therefore, we
may repeat the steps of the previous section to establish a holonomy
invariant measure and deduce all of the results of that section, up to and
including the fact that $f$ is expansive.  The result of Vieitez, however,
applies only in dimension three.  As such, we must establish the conjugacy
directly.

\begin{lemma}
    If $\tilde f:\Rd \to \Rd$ is a lift of $f$ to the universal cover, then for
    points $x,y \in \Rd$, $x=y$ if and only if the sequence
    $\|\tilde f^n(x) - \tilde f^n(y)\|$ is bounded for all $n \in \Z$.
\end{lemma}
\begin{proof}
    The system has Global Product Structure \cite{ham-thesis}.  As a
    consequence, for $x,y \in \Rd$ there are unique points $p,q \in \Rd$ such that
    $p \in \Wu(x)$, $q \in \Wc(p)$ and $y \in \Ws(q)$.

    Suppose $x  \ne  p$.  Then $d_u(\tilde f^n(x),\tilde f^n(p))$ grows
    exponentially fast as $n \to \infty$, and, as $\Wu$ is
    quasi-isometric by assumption, $\|\tilde f^n(x) - \tilde f^n(p)\|$ grows at
    the same rate.  In particular, by the definition of absolute partial
    hyperbolicity, this rate of growth is faster than in the stable or center
    directions, and so $\|\tilde f^n(x) - \tilde f^n(y)\|$ tends to infinity as
    well.

    Hence, we may assume $x = p$ and, by the same logic for the stable
    direction, that $q = y$.  We have reduced to the case where $x$ and $y$
    lie on the same center leaf.  If $\gamma$ is the center curve connecting
    these points, then using a holonomy invariant measure $\mu$ as in the
    previous section, we can show that $\mu(\tilde f^n \circ \gamma)$ is unbounded
    for $n \in \Z$, and hence that $d_c(\tilde f^n(x), \tilde f^n(y))$ is
    unbounded as well. The center foliation is quasi-isometric
    \cite{ham-thesis}, from which the result follows.
\end{proof}
\begin{prop}
    $f$ is conjugate to Anosov.
\end{prop}
\begin{proof}
    As $f$ is homotopic to the hyperbolic toral automorphism $A$,
    there is a semi-conjugacy, a surjective map $h:\Td \to \Td$ such that
    $h f(x)=A h(x)$ for all $x \in \Td$ \cite{Franks1}. We may assume additionally
    that $h$ is homotopic to the identity, so that $f$, $h$, and $A$, lift to
    maps on the universal cover where
    $\tilde h \circ \tilde f = A \circ \tilde h$ and
    $\|\tilde h(x) - x\|$ is bounded for all $x \in \Rd$.

    Suppose $\tilde h(x) = \tilde h(y)$ for points $x,y \in \Rd$.  By the
    semi-conjugacy, $\tilde h \tilde f^n(x) = \tilde h \tilde f^n(y)$ and
    $\|\tilde f^n(x) - \tilde f^n(y)\|$ is bounded for all $n \in \Z$.
    By the previous lemma, $x=y$.  That is, the semiconjugacy $\tilde h$ on
    $\Rd$ is injective, and therefore the semiconjugacy $h$ on $\Td$ is
    injective, making it a true conjugacy between $f$ and $A$.
\end{proof}
\section{Full Conjugacy} \label{conjsec}  

Consider a linear automorphism $A:\T^3 \to \T^3$ with eigenvalues
$\lambda_s < \lambda_c < 1 < \lambda_u$
and the corresponding partially hyperbolic splitting.  A small perturbation
of $A$ will be both partially hyperbolic and topologically conjugate to $A$.
However, the conjugacy may not preserve all of the partially hyperbolic
structure.  If the perturbation $f$ is accessible, there is no way that the
conjugacy can map both $\Ws_f$ and $\Wu_f$ to the corresponding one-dimensional
foliations of $A$.  In the special case, that $f$ is not accessible, the
conjugacy is as good as possible, preserving all of the invariant foliations.

\begin{thm} \label{fullconj}
    Suppose $f:\T^3 \to \T^3$  is an absolutely partially hyperbolic
    diffeomorphism,
    $\Es_f \oplus \Eu_f$ is integrable, and $h:\T^3 \to \T^3$ is a conjugacy from
    $f$ to a linear hyperbolic toral automorphism $A:\T^3 \to \T^3$.  Then $A$
    is partially hyperbolic with a linear splitting such that
    $W^*_A = h(W^*_f)$
    for $* = c,s,u,cs,cu,us$.
\end{thm}
\begin{remark}
    The following proof generalizes to the case of a $d$-dimensional torus,
    $d  \ge  3$, under the additional assumptions that $\Wu_f$ and $\Ws_f$ are
    quasi-isometric, and $\dim \Ec_f = 1$.  As such, we use $\Td$ in place of
    $\T^3$ throughout the proof.
\end{remark}
\begin{proof}
    Assume without loss of generality that $f_* = A_*$ as automorphisms
    of $\pi_1(\Td)$, and that $h$ is homotopic to the identity.  It follows
    from Section 2 of \cite{ham-thesis}, that $\Ec_f$ is quasi-isometric and
    that $f$ and $A$ are absolutely partially hyperbolic and with the same
    choice of constants in the definition.  In particular, there are constants
    $\gammahat < 1 < \gamma$ such that for points $x$ and $y$ on the universal
    cover,
    \begin{align*}
        y &\in \Wc_f(x)  \quad \Leftrightarrow \\  
        C \inv \gammahat^n  &\le  \|f^n(x) - f^n(y)\|  \le  C \gamma^n
        \quad \text{for some $C > 1$ and all $n \in \Z$}  \quad \Leftrightarrow \\  
        C \inv \gammahat^n  &\le  \|A^n(h(x)) - A^n(h(y))\|  \le  C \gamma^n
        \quad \text{for some $C > 1$ and all $n \in \Z$}  \quad \Leftrightarrow \\  
        h(y) &\in \Wc_A(h(x)),
    \end{align*}
    where, by abuse of notation, we use $f$, $A$, and $h$ to represent the
    lifted maps on $\Rd$.  We have shown $\Wc_A = h(\Wc_f)$.

    Assume, without loss of generality, that $A$ has a contracting center
    direction.  That is, $\Ec_A \oplus \Es_A$ corresponds to the
    stable direction of the Anosov splitting.
    Then, by considering
    $\|f^n(x) - f^n(y)\|$ as $n \to \infty$ or $n \to -\infty$, one shows that
    $\Wcs_A = h(\Wcs_f)$ and $\Wu_A = h(\Wu_f)$.  It remains to show $\Ws_A =
    h(\Ws_f)$.  On the universal cover, fix the leaf $L$ of $\Wcs_A$ passing
    through the origin.  For each $z \in \Z^d$, define the map $\tau_z:L \to L, x
    \mapsto \Wu_A(x+z) \cap L$.  Each $\tau_z$ is a translation on $L$.  The
    unstable foliation of an Anosov map on a torus $\Td$ is minimal
    \cite{Franks1}.  Therefore, the set of translations $\{\tau_z : z \in \Z^d
    \}$ is dense in the set of all rigid translations of $L$.

    As $\Wu_A=h(\Wu_f)$ and $h(\Ws_f)$ subfoliate $h(\Wus_f)$, one can verify that
    the translations $\tau_z$ preserve the restriction to $L$ of the foliation
    $h(\Ws_f)$.  By continuity of the foliation, \emph{every} rigid translation
    defined on $L$ preserves $h(\Ws_f)$.  Then, a leaf of $h(\Ws_f)$ is a
    locally compact, $C^1$-homogeneous subset of $L$, and is therefore a
    $C^1$-manifold \cite{repovs1996}.  The tangent spaces to the leaves are
    also invariant under every translation of $L$ and therefore the foliation
    is linear.

    Quotienting down to $\T^d$, the image of $L$ is dense and so $h(\Ws_f)$ is
    linear on all of $\T^d$.  As a linear invariant foliation, $h(\Ws_f)$ must
    equal $\Ws_A$.
\end{proof}
\section{The central Lyapunov exponent} \label{expt} 

A conservative diffeomorphism is \emph{weakly ergodic} if almost every point
has a dense orbit.

\begin{thm} \label{weak}
    If $f:\bbT^3 \to \bbT^3$ is a conservative partially hyperbolic $C^2$
    diffeomorphism homotopic to Anosov, then it is weakly ergodic.
\end{thm}
\begin{proof}
    If $f$ is ergodic, it is weakly ergodic.
    If $f$ is not ergodic, all of the results of Section
    \ref{mainproof} hold.  In particular, each accessibility class is the leaf
    of a minimal foliation (Propositions \ref{usint} and \ref{usmin}) and is
    therefore dense.  It then follows from the work of Burns, Dolgopyat, and
    Pesin that almost every orbit is dense \cite[Lemma 5]{bdp}.
\end{proof}
\begin{thm}
    Suppose $f:\T^3 \to \T^3$ is a conservative partially hyperbolic $C^2$
    diffeomorphism homotopic to Anosov.   If $f$ is not ergodic, then the
    central Lyapunov exponent is zero almost everywhere.
\end{thm}
\begin{proof}
    To prove the contrapositive, assume the central Lyapunov exponent is
    non-zero on a positive measure subset.
    By the work of Burns, Dolgopyat, and Pesin, there is an ergodic component
    of $f$ which is open mod zero \cite[Theorem 1]{bdp}.  That is, there is
    $X \subset \bbT^3$ invariant such that $f|_X$ is ergodic and $X$
    differs from a non-empty open set $U \subset \bbT^3$ by a set of measure
    zero.  By Theorem \ref{weak}, the $f$-saturate of $U$ has full measure.
    Then, $X$ also has full measure which means $f$ is ergodic.
\end{proof}
\section{3-normal hyperbolicity} \label{three-nh} 

Suppose $f:\bbT^3 \to \bbT^3$ is an absolutely partially hyperbolic
system satisfying the hypotheses of Theorem \ref{fullconj}.
Then, for a center leaf $\Wc(x_0)$ on the universal cover $\bbR^3$
and a deck translation $\tau:\bbR^3 \to \bbR^3$ in $\pi_1(\bbT^3)$, there is a
unique map $h_\tau: \Wc(x_0) \to \tau(\Wc(x_0))$ defined by
\[
    h_\tau(x) \in \tau(\Wc(x_0)) \cap \Wus(x).
\]
Since $\Wcs(x_0)$ and $\Wcu(\tau(x_0))$ intersect in a unique center leaf,
the map $h_\tau$ can be viewed as a stable holonomy inside a center-stable leaf
followed by an unstable holonomy inside a center-unstable leaf.  In the case
of a one-dimensional center, both of these holonomies are $C^1$ \cite{PSW}.
Therefore, $h_\tau$ is $C^1$.  Further, these maps define a free action of the
fundamental group on the center leaf{:}
\[
    \alpha: \pi_1(\bbT^3) \to \Diff^1(\Wc(x_0)),
    \quad \alpha(\tau) = \tau \inv \circ h_\tau.
\]
Regrettably, a $C^1$ action is not enough to prove ergodicity.
However, a $C^2$ action is.

\begin{lemma} \label{ergaction}
    If $\alpha(\tau)$ (as defined above) is $C^2$ for each $\tau \in \pi_1(\bbT^3)$,
    then $f$ is ergodic.
\end{lemma}
\begin{proof}
    We show that $f$ is essentially accessible, that is, every
    measurable $us$-saturated subset of $\bbT^3$ has either zero measure or
    full measure.
    This is enough to prove ergodicity \cite{BW-annals}.

    Suppose $A \subset \bbT^3$ is such a set and
    $\tilde A$ is its lift to the universal cover.
    As the $\Wu$ and $\Ws$ foliations are absolutely continuous,
    the intersection
    $B := \tilde A \cap \Wc(x_0)$
    has zero or full measure if and only if $\tilde A$ does.
    Note that $\alpha(\tau)(B) = B$ for every $\tau$.
    Take elements $\tau_1, \tau_2 \in \pi_1(\bbT^3)$ such that
    $\langle \tau_1, \tau_2 \rangle$ is isomorphic to $\bbZ^2$.
    Then $C := \Wc(x_0) / \alpha(\tau_1)$ is a $C^2$ manifold homeomorphic to a
    circle, and $\alpha(\tau_2)$ defines a $C^2$ diffeomorphism of this manifold
    with an irrational rotation number.  In such a case, every measurable,
    $\alpha(\tau_2)$-invariant subset of $C$ must have either zero Lebesgue
    measure or full Lebesgue measure
    \cite[Th\'eor\`eme 1.4]{herman1979conjugaison}.
    In particular, $B$ has zero measure or full measure and the claim is
    proved.
\end{proof}
To apply this lemma, it is enough to have a point $x_0$ such that the
manifolds $\Wcs(x_0)$ and $\Wcu(x_0)$ are $C^2$ and
such that the stable/unstable holonomies between center leaves inside these
manifolds are $C^2$.
In fact, if we assume that $\Wcs(x_0)$ and $\Wcu(x_0)$
are $C^3$, the $C^2$ regularity of the holonomies follows.

To see this, first recall the definition of the \emph{norm} $\|A\| =
\sup_{\|v\|=1} \|A v\|$ and \emph{co-norm} $m(A) = \inf_{\|v\|=1} \|A v\|$ of a
linear operator $A:V \to W$ between normed vector spaces.  For a partially
hyperbolic system, take $T^s_x f$ to signify $T_x f|_{\Es(x)}$ and similarly for
the superscripts $c$ and $u$.

\begin{thm}
    [Pugh-Shub-Wilkinson \cite{PSW}] \label{psw}
    Let $f:M \to M$ be a $C^{r+1}$ partially hyperbolic
    diffeomorphism, $r  \ge  1$, such that
    \[
        \|T^s_xf\| \|T^c_xf\|^r < m(T^c_xf)
    \]
    for all $x \in M$.  Then, stable holonomies are $C^r$ smooth on any
    $C^{r+1}$ center-stable leaf.
\end{thm}
\begin{cor} \label{jacobian}
    Let $f$ be a volume-preserving partially hyperbolic $C^3$ diffeomorphism
    on a three-dimensional manifold.  Then, stable holonomies are $C^2$
    smooth on any $C^3$ center-stable leaf.
\end{cor}
\begin{proof}
    The Jacobian of $f^n$ satisfies
    \[
        1 = J_{f^n}(x) =
        \|T^s_x f^n\| \|T^c_x f^n\| \|T^u_x f^n\|
        \frac{V(x)}{V(f^n(x))}
    \]
    where $V:M \to \bbR$ is a continuous, positive function determined by the
    angles between $\Es$, $\Ec$, and $\Eu$.  Consequently, there is $C > 0$ such
    that
    \[
        \|T^s_x f^n\| \|T^c_x f^n\| \|T^u_x f^n\| < C
    \]
    for all $x$ and $n$.  For $n$ sufficiently large,
    $
        \|T^s_xf^n\| \|T^c_xf^n\|  \le  \|T^u_xf^n\| \inv C < 1
    $
    and since $\Ec$ is one-dimensional, $\|T^c_xf^n\| = m(T^c_xf^n)$ and
    Theorem \ref{psw} is satisfied for $f^n$ and $r = 2$.
\end{proof}
Therefore, to prove ergodicity, we need only find center-stable and
center-unstable leaves which are $C^3$.

A partially hyperbolic $C^r$ diffeomorphism $f:M \to M$ is \emph{$r$-normally
hyperbolic} if it is dynamically coherent, and
\[
    \|T^s_xf\| < m(T^c_xf)^r  \le  \|T^c_xf\|^r < m(T^u_xf)
\]
for all $x \in M$.

If $f$ is $r$-normally hyperbolic, its $cs$- and $cu$-leaves are $C^r$
\cite[Theorem 6.1]{HPS}.

\begin{thm}
    If $f:\bbT^3 \to \bbT^3$ is volume-preserving,
    absolutely partially hyperbolic,
    $3$-normally hyperbolic, and
    homotopic to Anosov,
    then it is ergodic.
\end{thm}
\begin{remark}
    These last two results need the condition of ``volume-preserving'' as
    opposed to just ``conservative'' as defined in Section \ref{definitions}.
    This is because the proof of Corollary \ref{jacobian} uses that the
    Jacobian is equal to one.
\end{remark}
\begin{remark}
    In this section, we assumed \emph{absolute} partial hyperbolicity for
    convenience.  With some work, the above proof can be generalized to
    \emph{point-wise} partially hyperbolic systems which are dynamically
    coherent.  Further, R.~Potrie has announced that any point-wise partially
    hyperbolic system $f:\bbT^3 \to \bbT^3$ homotopic to Anosov is dynamically
    coherent.
\end{remark}
\bigskip

\acknowledgement
The authors are grateful to Federico and Jana Rodriguez Hertz for many helpful
discussions.
R.U.~was partially supported by ANII grant FCE\_2009\_1\_2687 and
A.H.~was partially supported by CNPq (Brazil).


\bibliographystyle{plain}
\bibliography{dynamics}

\medskip

\noindent {\sc The University of Sydney and the University of New South Wales}
\\
{\em Mailing address:}\\
\indent Department of Mathematics, University of Sydney, NSW 2006, Australia.\\
{\em E-mail address:} {\tt andy@maths.usyd.edu.au}

\medskip

\noindent {\sc Universidad de la Rep\'ublica}
\\
{\em Mailing address:}\\
\indent IMERL, Facultad de Ingenier\'ia, CC 30, Montevideo, Uruguay.\\
{\em E-mail address:} {\tt ures@fing.edu.uy}

\end{document}